\documentclass{article}

\usepackage{amsmath, amsthm, amsfonts}

\newtheorem{thm}{Theorem}[section]
\newtheorem{corollary}[thm]{Corollary}
\newtheorem{lem}[thm]{Lemma}

\theoremstyle{definition}
\newtheorem{definition}[thm]{Definition}
\theoremstyle{remark}
\newtheorem{rem}[thm]{Remark}
\theoremstyle{axiom}
\newtheorem{axiom}[thm]{axiom}

\title{On zero dimensional sequential spaces}

\author{Paul Fabel\\
  \small Department of Mathematics \& Statistics\\
  \small Mississippi State, MS,39762\\
  \small apf3@msstate.edu
}
\date{}

\begin{document}

\maketitle

\abstract{We develop tools to recognize sequential spaces with large inductive
  dimension zero. We show the Hawaiian earring group $G$ is 0 dimensional,
  when endowed with the quotient topology, inherited from the space of based
  loops with the compact open topology. In particular $G$ is $T_4$ and hence
  inclusion $G \hookrightarrow F_M (G)$ is a topological embedding into the
  free topological group $F_M (G)$ in the sense of Markov.}

\renewcommand{\abstractname}{Acknowledgements}
\begin{abstract}
 The author gratefully acknowledges partial support from University
Primorska, Feb 22-August 1 2020
\end{abstract}



\section{Introduction}

When is a Hausdorff sequential space zero dimensional? The fundamental group
of the Hawaiian earring serves as catalyst for such an inquiry in the context
of the following three questions. We answer the first question affirmatively
via partial answers to second and third.

1) Is the large inductive dimension of the Hawiian earring group zero, if
$G$ enjoys the quotient topology inherited from the space of based loops?

2) If a sequential space $G$ continuously injects into a countable inverse of
limit of discrete spaces, what conditions ensure $G$ is zero dimensional?

3) If a sequential space $G$ is a quotient of a countable product of discrete
spaces, what conditions ensure $G$ is zero dimensional?

Wild algebraic topology is loosely described as the study of locally complicated
spaces, and their attendant homotopy/homology groups. The motive to impose a
topology on the latter objects might come from functorality of the fundamental
group {\cite{B}}, from a canonical bijection between $\pi_1 (X, p)$ and the
fibres of a semicovering $E \rightarrow X$
{\cite{FZ2}}{\cite{VZ2}}{\cite{VZ3}}, or to measure the extent to which
$\pi_1$ might act continuously on a space {\cite{B1}}.

At center stage
{\cite{E0}}{\cite{CC1}}{\cite{E1}}{\cite{Cors1}}{\cite{CE}}{\cite{CK1}} is the
Hawaiian earring ${HE}$, a null sequence of loops joined at a common
point, the inverse limit of nested sequence of bouquets on $n$ loops, under
retraction bonding maps, collapsing the ${nth}$ loop to the special point
$p.$

The induced homomorphism $\phi : \pi_1 ({HE}, p) \rightarrow
\lim_{\leftarrow} F_n$ , with $F_n$ the discrete free group on $n$ generators,
is one to one {\cite{DeSmit}}{\cite{MM}}{\cite{FZ}}. Thus the subgroup
${im} (\phi)$ determines a sense in which elements of $\pi_1 ({HE},
p)$ can be understood as precisely the ``infinite irreducible words'' in the
letters $x_1, x_1^{-1}, x_2 \ldots .$ so that each letter appears finitely
many times {\cite{CC1}}.

To impose a topology on $\pi_1 ({HE}, p)$, at one extreme we might insist
that $\phi$ is a topological embedding. This creates a zero dimensional
topological group, but with the drawback that the topology is permissive for
what is allowed to converge. For example the sequence $(x_1 x_n x_1^{- 1}
x_n)^n \rightarrow 0 \in \pi_1 ({HE}, p),$but all corresponding lifts
diverge, with the topology of uniform convergence.

At another extreme, invoking a construction similar to the familiar universal
cover, {\cite{Bog}}{\cite{FZ}}, the group $\pi_1 ({HE}, p)$ acts freely
and isometrically on a corresponding generalized universal cover of
${HE}$, a uniquely arcwise connected, locally path connected metric
space, i.e. a topological R-tree $R$. The trade off here is that it is
difficult for our isometries to converge. Treated as a group of isometries of
$R$, $\pi_1 ({HE}, p)$ fails to be even a quasitopological group, we can
have $x_n \rightarrow {id}$ with $\{ x_n x_1 \}$ diverging.

One compromise is to impose the quotient topology on $\pi_1 ({HE}, p),$
defined as a quotient of the space of based loops in ${HE}$ with the
uniform topology. This resolves two of the mentioned drawbacks, the sequence
$\{ (x_1 x_n x_1^{- 1} x_n)^n \}$ now diverges, and group translation is now
continuous. On the other hand this comes at the cost of metrizability. The
subspace $\{ 0, (x_1 x_n x_1^{- 1} x_n)^m \} \subset \pi_1 ({HE}, p)$ is
a Frechet Urysohn fan {\cite{A1}} {\cite{FA1}} and hence $\pi_1 ({HE},
p)$ is not first countable. This begs the question of which familiar
separation axioms does $\pi_1 ({HE}, p)$ satisfy?

Continuity of the injection $\phi : \pi_1 ({HE}, p) \rightarrow
\lim_{\leftarrow} F_n $ensures $\pi_1 ({HE}, p)$ is $T_2$, since the
codomain is $T_2$. Morever $\pi_1 ({HE}, p)$ is a quotient of a separable
metric space and hence $\pi_1 ({HE}, p)$ is Lindelof. Unfortunately
$\phi$ is {{not}} a topologial embedding {\cite{FA0}}, and $\pi_1
({HE}, p)$ is {{not}} a topological group in TOP {\cite{FA1}}. This
calls into question whether $\pi_1 ({HE}, p)$ is at least $T_3$, and
hence $T_4,$ since $\pi_1 ({HE}, p)$ is a Lindelof space. To prove $\pi_1
({HE}, p)$ is $T_4$ it suffices to prove $\pi_1 ({HE}, p)$ has large
inductive dimensions zero, that disjoint closed sets can be thickened into
disjoint clopens.

While the class of contractible space shows dimension is generally not an
invariant of the homotopy type of an underlying space, functorality ensures
$\pi_1 (X)$ and $\pi_1 (Y)$ have the same dimension if $X$ and $Y$ are
homotopy equivalent. More esoterically, the knowledge that $\pi_1 ({HE},
p)$ is zero dimensional will ensure for example, that $\pi_1 ({HE}, p)$
cannot contain a copy of the totally disconnected 1 dimensional Erdos space
{\cite{D1}}.

To prove $\pi_1 ({HE}, p)$ is zero dimensional we establish Theorems
\ref{gzero} and \ref{mainapp}, applicable to suitably well behaved quotients
of the inverse limit of countably many discrete nested retracts.

The potential difficulty of such an inquiry is highlighted by the familiar
dimension raising closed quotient of the Cantor set $\{ 0, 1 \} \times \{ 0, 1
\} \ldots \rightarrow [0, 1]$ mapping each binary sequence onto the
corresponding real number. What goes wrong? All finite approximations to $(0,
1, 1, .)$ and $(1, 0, 0 \ldots)$ are distinct, yet the points are identified
in the limit. This is analogous to a failure of $\pi_1$ injectivity in shape
theory. To avoid this, in this paper we only consider quotients where the
above phenomenon does not happen, and in particular the above does not happen
in the Hawaiian earring.

Most of the paper is devoted to a proof of the following. Suppose $X_1
\subset X_2 \ldots$ is a nested sequence of discrete retracts $X_{n + 1}
\rightarrow X_n $ and $q_n : X_n \rightarrow G_n$ is a quotient map. This data
induces a quotient map $q : \lim_{\leftarrow} X_n \rightarrow G$, and the
following two extra assumptions ultimately ensure that $q$ does not raise
dimension. 1) The retractions $X_{n + 1} \rightarrow X_n$ induce a map $G_{n +
1} \rightarrow G_n$ and 2) The inclusion maps $X_n \rightarrow X_{n + 1}$
induce a map $G_n \rightarrow G_{n + 1}$. Our main result (Theorems
\ref{gzero} and \ref{mainapp}) is that a space $G$ constructed in this manner
has large inductive dimension zero. The proof uses the well ordering
principle, and we now indicate why well ordering is useful to circumvent the
failure of a less sophisticated approach.

The proof idea for Theorems \ref{gzero} and \ref{mainapp} stems from an easier
but still complicated construction designed to prove that the Hawaiian earring
group $\pi_1 ({HE}, p)$ is a $T_3$ space. To prove that $\pi_1
({HE}, p)$ is a $T_3$ space it suffices to prove that $\pi_1 ({HE},
p)$ has a basis of clopen sets, i.e. that $\pi_1 ({HE}, p)$ has small
inductive dimension 0. In turn, since $\pi_1 ({HE}, p)$ is homogeneous,
it suffices, to start with a closed set $B \subset \pi_1 ({HE}, p)$ so
that the identity $e \notin B,$ and thicken $B$ into a clopen set $U (B)$ so
that $B \subset U (B)$ and $e \notin U (B)$. We now outline the overall strategy
for building $U (B)$ and point out a potenially fatal pitfall.

By construction $\pi_1 ({HE}, p) = G$ is equipped with a canonical
countable collection of clopen sets $U (g_n) \subset \pi_1 ({HE}, p)$,
the preimage of $g_n \in G_n$ under the retraction $G \rightarrow G_n,$ here
$G_n$ is the free group on $n$ generators. The overall idea it to somehow
thicken $B \subset \pi_1 ({HE}, p)$ into a union $U (B)$ of our special
clopens, so that $U$ is clopen and $e \notin B$.

Given a closed set $B \subset \pi_1 ({HE}, p)$, the simplest idea to
construct $U (B)$ \ is to exploit the God given retraction $R : G\backslash \{
e \} \rightarrow \left( \bigcup G_n \right) \backslash \{ e \}$. Unfortunately
this does not work, as indicated below, suggesting why more refined methods
for building $U (B)$ are needed.

Consider the retraction $R : G\backslash \{ e \} \rightarrow \left( \bigcup
G_n \right) \backslash \{ e \}$ taking the nontrivial word $g \in G$ to $R (g)
= g_n$ (deleting all letters greater than $n$ from $g$) with $n$ minimal so
that $g_n$ is nontrivial. Given $B \subset G$ closed with $e \notin B$ it is at
least plausible that $\bigcup_{g \in B} U (R (g))$ is clopen, but the
following example shows this is false, $e$ can be a missing limit point.
Suppose for $k > 1$ the closed set $B$ is the sequence of finite words $\{ w
(k) \}$ with \ $w (k) = (x_1 x_{k + 1} x_1^{- 1} x_{k + 1})^k x_k .$ Thus $R
(w (k)) = x_k$ and hence $e$ is a missing limit point of the union of the sets
$U (x_k) .$ \

For the latter example the reader might notice that we could get an
acceptable thickening $U (B)$ using the union of the sets $U (x_1 x_{k + 1}
x_1^{- 1} x_{k + 1})^k x_k),$ i.e. given $b \in B$ we should look for a large
index approximation ${of}$b, rather than a small index approximation of
$b.$ This is indeed the right idea, but comes at the cost of no obvious best
method to approximate $b$, we have too many choices as shown by the example $b
= x_1 x_2 x_3 \ldots . \in \pi_1 ({HE}, p)$. In other words the best way
to approximate $b \in B$ is context dependent, depending both on $B$ and also
previously made choices when attempting to thicken some of $B$.

To make the previous sentence more precise, we indicate a more systematic way
to thicken $B \subset \pi_1 ({HE}, p)$ into a clopen. The overarching
idea is to impose a linear order on $\pi_1 ({HE}, p)$ (not compatible
with its topology, and not well ordered), but so that nevertheless each closed
set $A \subset \pi_1 ({HE}, p)$ has a minimal element. Given such an
ordering, we start with the minimal $b_0 \in B,$ then thicken $b_0 $ into $U
(R (b_0))$, then let $b_1$ the minimal element of the closed set $B\backslash
U (R (b_0)) .$ Then we select $U (R (b_1))$ and so on. Crucially we hope to
ensure the union of the sequence $U (R (b_0)) \bigcup U (R (b_1)) \bigcup
\ldots$is clopen.

If we are suitably careful with the definition of our linear order on $G$,
Lemma \ref{induct} ensures that, proceeding by transfinite recursion, the
union of our selected sets $U (R (b_i))$ is indeed the desired clopen set $U
(B) .$ This idea is the key to the paper, and we hope it will remain
undisguised by the superficially technical appearance of its implementation.

Corollary \ref{he} shows $\pi_1 ({HE}, p)$ is indeed 0 dimensional and in
particular $\pi_1 ({HE}, p)$ is a Tychonoff space (completely regular).
This pays off categorically both in TOPGRP {\cite{AT1}} and the (compactly
generated groups) $k -$GRP {\cite{Po}}.

On the one hand we might ingressively blame the failure of $\pi_1 ({HE},
p) $ to be a topological group in TOP on an abundance of closed sets (Theorem
1 {\cite{Ba0}}). As summarized in section 3 {\cite{B}}, can $\pi_1 ({HE},
p)$ be repaired categorically, by coarsening $\pi_1 ({HE}, p) $to be the
canonical quotient of the free Markov topological group $F_M (\pi_1
({HE}, p))$. The payoff here is that our new knowledge that $\pi_1
({HE}, p)$ is a Tychonoff space \ contributes to our understanding of the
latter construction. We are assured that inclusion $\pi_1 ({HE}, p)
\rightarrow F_M (\pi_1 ({HE}, p))$ is a topological embedding, as noted
in Lemma 3.1 {\cite{B}}. See also {\cite{AT1}}.

More congressively {\cite{ChE}}, instead of coarsening $\pi_1 ({HE}, p),$
we might instead accept $\pi_1 ({HE}, p)$ as a 0-dimensional first class
citizen in $k - {SEQGROUP}$, the category of sequential spaces with group
structure, so that the group operations are sequentially continuous, and
acknowledge that the familiar product topology is categorically not always the
most useful way to multiply spaces {\cite{BR1}} {\cite{BT}}.

Looking ahead the hope is that the tools developed in this paper will prove
useful to answer more general questions such as ``If $X$ is a planar
continuum, with the quotient topology must $\pi_1 (X, p) $ be zero
dimensional?'' We conjecture the answer is ``yes''.

\section{Three axioms ensure $G$ is zero dimensional}

The purpose of this section is to prove Theorem \ref{gzero}, every space $G
\quad$satisfying the three axioms below has large inductive dimension zero.

We do not assume $G$ is the Hawaiian earring group, but the reader may find
it helpful to assume otherwise, to assume $G_n$ is the free group of maximally
reduced words on $n$ generators, and to treat a nontrivial element $g \in G$
as an ``irreducible'' countably infinite word in the letters $x_1, x_1^{- 1},
x_2, \ldots$ so that each letter appears finitely many times, and so that each
subinterval of letters in $g$ represents a nontrivial loop in the Hawaiian
earring.

Constructing a useful linear order on $G$ is a complicated affair carried out
in detail in section \ref{apps}, but the rough idea is the following. Assuming
$G$ is the Hawaiian earring group, to compare $g \in G$ and $h \in G$, first
delete all letters except $x_1$ and $x_1^{- 1},$ but {{don't}} reduce. If
the surviving unreduced words are different this is adequate to tell $g$ and
$h$ apart. 
We can also arrange that homotopy classes of words in $x_1,x_1^-1$ determine non interlaced subsets of $G$. 
Having well ordered
the unreduced words in $ x_1, x_1^- 1,$ we then extend the well
ordering to the unreduced words in $\{ x_1, x_1^{- 1}, x_2, x_2^{- 1} \}$
so as to ensure axioms \ref{lin} and \ref{trap} are destined to hold. The
important point is that to compare $g$ and $h$ we look, successively at their
{{unreduced}} approximations, until we find the minimal index where they
differ as unreduced words. In particular $G$ does {
{not}} have the
lexigraphic order determined by $G_n .$

\

\begin{axiom}
  \label{retracts}We assume $G$ is a sequential space (a space so that if $A
  \subset G$ is not closed, there exsts a convergent sequence $a_n \rightarrow
  x$ so that $a_n \in A$ and $x \notin A$). We assume $G_1 \subset G_2$.... is a
  nested sequence of closed discrete subspaces and for each $n \in \{1, 2,
  3..\}$ the map $\Pi_n : G \rightarrow G_n$ is a retraction. We assume the
  canonical map $\phi : G \rightarrow \Pi G_n$ is one to one, defined as $\phi
  (g) = \{\Pi_n (g)\}$. We we do NOT require that $\phi$ is a topological
  embedding and we do NOT require that $\phi$ is a surjection. We assume
  $\Pi_{n - 1} = \Pi_{n - 1} \Pi_n$ and we assume the sequence $\Pi_n (g)
  \rightarrow g$ pointwise for each $g \in G$.
\end{axiom}

\begin{axiom}
  \label{lin}We assume $G$ admits a linear order $<$ so that every closed set
  $B \subset G$ has a minimal element, so that $\Pi_1 (g) \leq \Pi_2 (g) ....
  \leq g$, so that each subspace $G_n$ is well ordered, and for each strictly
  increasing sequence $g_1 < g_2$.... in $G$, either every subsequence of
  $\{g_n \}$ diverges, or $\lim g_n = \sup \{g_n \}$. (We do NOT require that
  $G, <$ has the order topology, and we do NOT assume that the discrete
  subspace $G_n$ has the order topology).
\end{axiom}

\begin{axiom}
  \label{trap}Define $G_{\infty} = G_1 \cup G_2, \ldots$ and given $k \in
  G_{\infty}$ obtain $N$ minimal so that $k \in G_N$ and define ${Blowup}
  (k) = \Pi_N^{- 1} \Pi_N (k)$. We assume if $k_1 < k_2 < k_3$ with each $k_n
  \in G_{\infty}$ then if ${Blowup} (k_3) \subset {Blowup} (k_1)$
  then ${Blowup} (k_2) \subset {Blowup} (k_1)$.
\end{axiom}

\subsection{Basic Lemmas}

\begin{rem}
  \label{gt2}Since each space $G_n$ is discrete the countable product $\Pi_{n
  = 1}^{\infty} G_n$ is $T_2$. Hence, $G$ is also $T_2$ since $\phi :
  G_{\infty} \rightarrow lim_{\leftarrow} G_n$ is continuous and one to one
  (altough typically NOT a topological embedding). In particular convergent
  sequences in $G$ have unique limits.
\end{rem}

\begin{lem}
  \label{basicgood}The restriction $\phi$\textbar$G_{\infty}$ maps
  $G_{\infty}$ bijectively onto the eventually constant sequences in
  $\lim_{\leftarrow} G_n .$ (In general $\phi$\textbar$G_{\infty}$is NOT a
  topologicl embedding.)
\end{lem}

\begin{proof}
  Given $g \in G_{\infty}$ obtain $M$ minimal so that $g \in G_M$. If $M \leq
  M + n$ then $G_M \subset G_{M + n}$ , and since $\Pi_{M + n}$ is a
  retraction we have $\Pi_{M + n}  (g) = g$. Thus $\phi (g)$ is eventually
  constant.
  
  Conversely suppose $g \in G$ and $\phi (g)$ is eventually constant. Obtain
  $N$ minimal so that $\Pi_{N + n}  (g) = \Pi_N (g)$ if $N \leq N + n$. By
  axiom\ref{retracts} $\Pi_n (g) \rightarrow g$ and by Remark\ref{gt2} $\Pi_N
  (g) = g. \quad$ Thus since$ \Pi_N$ is a retraction $g \in G_N$
  and hence $g \in G_{\infty}$.
\end{proof}

\begin{lem}
  \label{decent}Suppose $\{a, b\} \subset G$ and $a < b$. Then there exists
  $N$ so that if $N \leq n$, then $a < \Pi_n (b)$.
\end{lem}

\begin{proof}
  By axioms \ref{retracts} we know $\Pi_n (b) \rightarrow b$ and also $\Pi_1
  (b) \leq \Pi_2 (b)$.... $\leq b$. Thus if the result were false, and since
  $<$ is a linear order, we would have $\Pi_n (b) \leq a$ for all $n$. Hence
  since $\sup b_n = b$ by axiom \ref{lin}, we would obtain the contradiction
  $b \leq a$.
\end{proof}

\begin{lem}
  \label{big}Suppose $V \subset G$ is clopen and $V < b$. Then there exists
  $N$ so that $V < \Pi_n (b)$ if $N \leq n$.
\end{lem}

\begin{proof}
  Since the sequence $\Pi_n \{b\}$ is nondecreasing (axiom \ref{lin}) it
  suffices to find $N$ so that V$< \Pi_N (b)$. To get a contradiction suppose
  no such $N$ exists. For each $n$ obtain $k_n \in V$ so that $\Pi_n (b) \leq
  k_n$. If there exists $k_N$ so that $\Pi_n (b) \leq k_N$ for all $n$, then
  $b \leq k_N$ by axiom \ref{lin}. But since $k_N \in V$, this would
  contradict the assumption that $V < b$. Thus no such $k_N$ exists and hence
  for each $N$ the inequality $\Pi_n (b) \leq k_N$ as only finitely many
  solutions. Hence, starting at $k_1$, we can recursively manufacture
  interleaved subsequences $k_1 < b_{n_1} \leq k_{n_1} < b_{n_2} \leq
  k_{n_2}$..... Thus by axiom \ref{lin} both subsequences converge to the same
  limit, and in particular $\{k_n \}$ has a subsequence converging to $b$.
  Since $V$ is closed, we get the contradiction $b \in V$.
\end{proof}

\begin{lem}
  \label{basic2}If $U \subset G$ is nonempty and clopen in $G$, then $minU \in
  G_{\infty}$. If the convergent strictly increasing sequence $g_1 < g_2 ...
  \rightarrow g$ then $g \notin G_{\infty}$ and in particular there exists $N$
  so that if $N \leq n < m$ then $\Pi_n (g) < \Pi_m (g) < g$.
\end{lem}

\begin{proof}
  Suppose $U \subset G$ is nonempty and clopen. By axiom \ref{lin} $minU$
  exists. Suppose $b \in U \setminus G_{\infty}$. Since $U$ is open and since
  $\Pi_n (b) \rightarrow b$ (axiom \ref{retracts}), obtain $N$ so that $\Pi_N
  (b) \in U$. By axiom \ref{lin} $\Pi_n (b) \leq b$ and since $\Pi_N (b) \in
  G_N \subset G_{\infty}$ and since $b \notin G_{\infty}$ we have $\Pi_N (b) <
  b$. Hence $b$ is not minimal in $U$.
  
  If $g \in G_N$, then since $\Pi_N \leq id|G_N$, $g$ is minimal in the clopen
  set $V = \Pi_N^{- 1} \Pi_N (g)$. Since $V$ is open and since $g$ is minimal
  in $V$ it is impossible that there exists a convergent sequence $g_1 < g_2
  .... \rightarrow g$.
  
  \ 
\end{proof}

\begin{definition}
  If $L$ and $H$ are linealry ordered sets a function $f : L \rightarrow 2^H$
  is \em{strictly increasing} if, given $s < t$ in $L$, if $x \in f (s)$
  and $y \in f (t)$, then $x < y$.
\end{definition}

\begin{lem}
  \label{ezinc}Let $S \subset 2^G$ denote the collection of clopen sets in $G$
  and suppose $[0, i)$ is an intital segment of the well ordered set $J$.
  Suppose $\gamma : [0, i) \rightarrow S$ is strictly increasing and suppose
  $V (j) = \cup_{k \leq j} \gamma (k)$ is clopen for each $j < i$. Then $V (i)
  = \cup_{k \leq i} \gamma (k)$ is missing at most one limit point $x$. If so,
  there exists an increasing $s_1 < s_2$... sequence terminal in $[0, i)$.
  Moreover for any terminal increasing sequence $t_1 < t_2 <$... in $[0, i)$,
  if $x_n \in \gamma (t_n)$ then $x_n \rightarrow x$ with $x = sup (V (i))$.
\end{lem}

\begin{proof}
  Note $V (i)$ is open in $G$. If $V (i)$ is not closed in $G$, then, since
  $G$ is a sequential space (axiom \ref{retracts}) suppose $\{k_n \} \subset V
  (i)$ is a convergent sequence so that $k_n \rightarrow x \notin V (i)$.
  Suppose $a_n$ is also a convergent sequence in $V_i$ with $a_n \rightarrow y
  \notin V (i)$. To prove $V (i)$ is missing precisely one limit point it
  suffices to show that $x = y$, since this will ensure $V (i) \cup \{x\}$ is
  sequentially closed in the sequential space $T_2$ space $G$ (axiom
  \ref{retracts}), and hence that $V (i) \cup \{x\}$ is closed in $G$.
  
  Note $\{k_n \}$ admits no constant subsequence since otherwise we get the
  contradiction $x \in \gamma (s)$ for some $s$. Thus we may refine so that
  the terms of $\{k_n \}$ and $\{a_n \}$ are distinct. Moreover since $[0, i)$
  is well ordered, and since $\gamma_i$ is increasing, we may further refine
  so that both sequences are strictly increasing. Thus wolog $k_n \in \gamma_i
  (s_n)$ and $a_n \in \gamma_i (t_n)$ with $s_1 < s_2$... and $t_1 < t_2$...,.
  
  Note $\{s_n \}$ is unbounded in $[0, i)$ since otherwise we get the
  following contradiction. Let sup$\{s_n \} = s < i$. Then $V (s)$ is a closed
  subspace of $G$ and hence $x \in V (s) \subset V (i)$. Thus both sequences
  $\{s_n \}$ and $\{t_n \}$ are unbounded in the well ordered set $[0, i)$ and
  hence the sequences are interlaced. It follows from axiom \ref{lin} that $x
  = y$ and $x = sup (V (i))$.
\end{proof}

\begin{lem}
  \label{induct}Suppose $[0, i)$ is an initial segment of the well ordered set
  $J$, suppose $S$ denotes the clopen subsets of $G$. Suppose the functions
  $\kappa : [0, i) \rightarrow G$ and $\gamma : [0, i) \rightarrow S$ are
  strictly increasing. Suppose is a function $K : [0, i) \rightarrow S.$
  Suppose $V (j) = \cup_{k \leq j} \gamma (k)$ is clopen in $G$ for each $j <
  i$. Suppose for each $j < i$ we have $\gamma (j) = (\Pi_{\eta (j)}^{- 1}
  \Pi_{\eta (j)} (\kappa (i))) \setminus K (j)$ with $K (j)$ clopen in $G$ and
  $\kappa (j) < K (j)$. Suppose given $j, \kappa (j)$ and $K (j)$, the index
  $\eta (j)$ is minimal to ensure that $\gamma | [0, j]$ is increasing.
  Suppose $K (j)$ is eventually constant, i.e. there is a clopen set $K
  \subset G$ and $\lambda_1 \in [0, i)$ so that $K (j) = K$ if $\lambda_1 \leq
  j < i$. Then $V (i) = \cup_{j < i} \gamma (j)$ is clopen in $G$.
\end{lem}

\begin{proof}
  Note $V (i)$ is open in $G$. If $[0, i)$ is bounded then $i$ has a
  predecessor $j = i - 1$ and thus $V (i) = V (j)$ is clopen by hypothesis.
  Hence assume $[0, i)$ is unbounded. To obtain a contradiction suppose $V
  (i)$ is not closed in $G$.
  
  Recall Lemma \ref{ezinc}, let $\{c\} = \overline{V_i} \setminus V_i$ with $c
  = sup (V (i))$. By Lemma \ref{ezinc} obtain a convergent sequence $s_1 <
  s_2$... terminal in $[0, i)$, and note $\kappa (s_n) \rightarrow c$. Since
  $\kappa (s_n) \notin K$ and since $K$ is clopen, $c \notin K$. By axiom
  \ref{lin} $c = sup (\kappa (s_n))$ and since eventually $\kappa (s_n) < K$
  we have $c < K$.
  
  Since $\gamma (\lambda_1) < c$ we apply Lemma \ref{big} and Lemma
  \ref{basic2} to obtain $N$ minimal so that if $N \leq n < m$ then, $\gamma
  (\lambda_1) < \Pi_N (c) \leq \Pi_n (c) < \Pi_m (c) < c$.
  
  Minimality of $N$ ensures $\Pi_{N - n} (c) < \Pi_N (c)$ if $n \geq 1$ and in
  particular $\Pi_{N - n} (c) \neq \Pi_N (c)$. Thus, recalling axiom
  \ref{trap} we have ${Blowup} (\Pi_N (c)) = \Pi_N^{- 1} \Pi_N (c)$.
  
  We next show $\Pi_N (c) \notin \gamma (s)$ for all $s$. Note by definition
  $\gamma (\lambda_1) < \Pi_N (c) < c < K$. Thus, since $\gamma$ is
  increasing, $\Pi_N (c) \notin \gamma (s)$ if $s \leq \lambda_1$. If $\lambda_1
  < s$ then $\gamma (s) =$ $\Pi_{n (s)}^{- 1} \Pi_{n (s)} (\kappa (s))
  \setminus K$. Thus, since $\Pi_N (c) < c < K$ if $\Pi_N (c) \in \gamma (s)$
  we would get the contradiction $c \in \gamma (s)$.
  
  Since $\kappa (s_1) < \kappa (s_2) ... \rightarrow c$, since $\gamma (s_1) <
  \gamma (s_2) \ldots,$and since $\Pi_N (c) < c$, obtain $M$ so that $\Pi_N
  (c) < min (\gamma (s_M))$ and also so that if $M \leq n$ then (by continuity
  of the map $\Pi_N$ with image in the discrete space $G_N$) we have $\Pi_N
  (\kappa (s_n)) = \Pi_N (c)$.
  
  In particular for all $n \leq N$ we have $\Pi_n (\kappa (s_M)) = \Pi_n (c)$.
  Consequently $N < n (s_M)$ since otherwise we get the contradiction $c \in
  \gamma (s_M)$.
  
  Recalling axiom \ref{trap}, minimality of $\eta (s_M)$ ensures
  ${Blowup} (\Pi_{\eta (s_M)} (\kappa (s_M))) = \Pi_{\eta (s_M)}^{- 1}
  \Pi_{\eta (s_M)} (\kappa (s_M))$ and thus, since $N < \eta (s_M)$ and since
  $\Pi_N (\Pi_{\eta (s_M)} (\kappa (s_M))) =$ $\Pi_N (c)$, we have
  ${Blowup} (\kappa (s_M)) \subset {Blowup} (\Pi_N (c))$.
  
  Recall we have shown $\Pi_N (c) \notin \gamma (s)$ for all $s$, and also that
  $\Pi_N (c) < min \gamma (s_M)$, and hence $\emptyset \neq \{\lambda \in [0,
  i) | \Pi_N (c) < min \gamma (\lambda)\}$. Thus, since $[0, i)$ is a well
  ordered set, obtain $\lambda_2^{} \in [0, i)$ minimal so that $\Pi_N (c) <
  min \gamma (\lambda_2^{})$. Hence, since ${Blowup} (\kappa (s_M))
  \subset {Blowup} (\Pi_N (c))$, and since $\Pi_N (c) < min (\gamma
  (\lambda_2^{}) \leq min (\gamma (s_M))$, it follows Lemma \ref{basic2} and
  axiom \ref{trap}, ${that} {Blowup} (min \gamma (\lambda_2^{}))
  \subset {Blowup} (\Pi_N (c))$.
  
  Recall the injection $\phi : G \rightarrow \lim_{\leftarrow} G_n$ from axiom
  \ref{retracts}, and note ${by} {Lemma} \ref{basicgood} \phi
  |G_{\infty}$ maps $G_{\infty}$ precisely onto the eventually constant
  sequences in $lim_{\leftarrow} G_n$. Hence, following the proof of Lemma
  \ref{basicgood}, given $g \in G_{\infty}$, to obtain ${Blowup} (g) =
  \Pi_M^{- 1} \Pi_M (g)$ defined in axiom \ref{trap} we use the minimal
  constant $M$ so $\Pi_M (g) = \Pi_{M + n} (g) $for all $n$. Thus, our
  knowledge that ${Blowup} (min \gamma (\lambda_2^{})) \subset
  {Blowup} (\Pi_N (c))$ ensures $\Pi_n (c) = \Pi_n (\kappa
  (\lambda_2^{}))$ for all $n \leq N$. However, since $\Pi_N (c) < min (\gamma
  (\lambda_2^{})) = \Pi_{\eta (\kappa \lambda_2^{})} (\kappa (\lambda_2^{}))$
  the latter set inclusion is proper and hence $N < \eta (\kappa
  (\lambda_2^{}))$.
  
  The contradiction is as follows. Given $\lambda_2^{}, \kappa (\lambda_2^{})$
  and $K$, the index $\eta (\lambda_2^{})$ was chosen minimal to ensure
  $\gamma | [0, \lambda_2^{}]$ is increasing. On the other hand $N < \eta
  (\kappa (\lambda_2^{}))$ would have been an admissible choice since $\gamma
  (s) < \Pi_N^{- 1} \Pi_N (\kappa (\lambda_2^{}))$ for all $s < \lambda_2^{}$.
\end{proof}

\begin{rem}
  \label{wclopen}Suppose $[0, i)$ is a linearly ordered set and $\gamma : [0,
  i) \rightarrow 2^G$ is increasing and suppose $\gamma (j)$ is clopen for
  each $j < i$. Suppose $V (j) = \cup_{k \leq j} \gamma (k)$ is clopen for
  each $j < i$. Then the set $W (j) = \cup_{k < j} \gamma (j)$ is clopen since
  $W (j) = V (j) \setminus \gamma (j)$, the difference of two clopens.
\end{rem}

\subsection{Thickening $B$ when $a < B$}

\begin{thm}
  \label{VaB}Suppose $a \in G$ and $a < B$ with $B \subset G$ a nonempty
  closed set. The following proof constructs a clopen set $V (a, B) \subset G$
  so that $a < V (a, B)$ and $B \subset V (a, B)$. Moreover if the convergent
  increasing sequence $a_1 < a_2 ... \rightarrow a$ there exists $M$ so that
  if $M \leq n$ then $V (a_m, B) = V (a, B)$.
\end{thm}

\begin{proof}
  Let $J$ be a well ordered set with minimal element $0$ so that $G_{\infty} <
  | J |$. By axiom \ref{lin} let $b_0 = minB$. Apply Lemma \ref{decent} and
  obtain $N$ minimal so that $a < \Pi_N (b_0)$. Define $\gamma (0) = \Pi_N^{-
  1} (\Pi_N (b_0))$. We will use Lemma \ref{induct} repeatedly, in the special
  case $K_j = \emptyset$ for all $j$.
  
  Let $S$ denote the clopen subsets of $G$. Suppose $i \in J$ and $\gamma |
  [0, i) \rightarrow S$ is strictly increasing, suppose for each $j \leq i$
  the set $V (j) = \cup_{k \leq j} \gamma (k)$ is clopen in $G$. Suppose for
  each $j < i$ we have $\gamma (j) = \Pi_{n_j}^{- 1} \Pi_{n_j} (c_j)$. Suppose
  given $j$ and $c_j$ the index $n_j$ is minimal to ensure that $\gamma | [0,
  j]$ is increasing. Thus by Lemma \ref{induct} the set $W (i) = \cup_{j < i}
  V (j)$ is clopen in $G$. If $B \subset W (i)$ let $V (a, B) = W (i)$ and we
  are done.
  
  Otherwise by axiom \ref{lin} let $b_i = min (B \setminus W (i))$ and by
  Lemma \ref{big} obtain $N$ minimal so that $W (i) < \Pi_N^{- 1} \Pi_N
  (b_i)$. Define $\gamma | [0, i] = \gamma | [0, i) \cup \{i, \Pi_N^{- 1}
  \Pi_N (b_i)\}$. If $\gamma (j)$ has benn defined we have $min \gamma (j) \in
  G_{\infty}$ by Lemma \ref{basic2}. Thus, since $| G_{\infty} | < | J |$ the
  transfinite recursive construction eventually terminates in the desired
  clopen set $V (a, B)$.
  
  Suppose the increasing sequence $a_1 < a_2 ... \rightarrow a$. Recall
  $\gamma (0) = \Pi_N^{- 1} \Pi_N (b_0)$ with $b_0$ minimal in $B$ and $N$
  minimal so that $a < \Pi_N^{- 1} \Pi_N (b)$. Note for $0 < j$ the definition
  of $\gamma (j)$ only depends on $\gamma (k)$ for $k < j$, and in particular
  the definiton of $\gamma (j)$ does not depend on $a$. Thus $V (a, B)$ is
  determined by the data $B$ and $\Pi_N^{- 1} \Pi_N (b_0)$. Hence $V (a_m, B)
  = V (a, B)$ provided $N$ is minimal so that $a_m < \Pi_N (b)$. Since $a <
  b_0$, by axiom \ref{lin} the convergent nondecreasing sequences $\{\Pi_m
  (a)\}$ and $\{\Pi_n (b_0)\}$ can only be interleaved for finitely many
  terms, and hence there exists $M$ so that if $M \leq m$ then $N$ is also the
  minimal solution to $\alpha_m < \Pi_N (b_0)$. Thus $V (a_m, B) = V (a, B)$
  if $M \leq m$.
\end{proof}

\begin{lem}
  \label{seqcont}Suppose $X$ is a space and the set $[0, i]$ is well ordered
  with the order topology and suppose there exists a terminal sequence $s_1 <
  s_2$.... in $[0, i)$ and suppose $f : [0, i] \rightarrow X$ is a (not
  necessarily continuous) function. Then the following are equivalent 1) $f$
  is at continuous at $i$, (i.e. for each open $U \subset X$ with $\kappa (i)
  \in U$, there exists $j_U$ so that $(j_U, i] \subset f^{- 1} (U)$) 2) If the
  sequence $t_1 < t_2$.... is terminal in $[0, i)$ then $f (t_n) \rightarrow f
  (i)$. As a special case, if $f| \{t_n \}$ is eventually constant for each
  terminal increasing sequence $t_1 < t_2$... then $f| [0, i)$ is eventually
  constant.
\end{lem}

\begin{proof}
  $1 \Rightarrow 2$. Given $\kappa (i) \in U$, get the mentioned $j_U$, and
  note with finitely many exceptions we have $j_U < t_n$. $2 \Rightarrow 1$ To
  obtain a contradiction suppose not $1$. Get an open $U \subset X$ with
  $\kappa (i) \in U$ but so that $V = f^{- 1} (U)$ contains no open right ray.
  Starting with $n = 1$ and proceeding recursively, for each $n$ obtain $t_n >
  s_n$ and $t_n > t_{n - 1}$ so that $t_n \notin V$. Thus $t_1 < t_2$. and
  $\{t_n \}$ is terminal in $[0, i)$. Thus with finitely many exceptions $f
  (t_n) \in U$ and hence $t_n \in V$ eventually. This constradicts $t_n \notin
  V$.
\end{proof}

\

\subsection{$G$ has large inductive dimension 0}

\

\begin{thm}
  $\label{gzero} G$ has large inductive dimension $0$.
\end{thm}

\begin{proof}
  Let $J$ be a well ordered set with minimal element $0$ so that $| J
  |${\bigskip}$\geq | G_{\infty} |$. Suppose $A$ and $B$ are disjoint nonempty
  closed sets in $G$. Define $\kappa (0) = min (A \cup B)$. If $\kappa (0) \in
  A$ define $B (0) = B$, and define $K (0) = V (\kappa (0), B (0))$, the
  clopen set from Theorem \ref{VaB}. \ If $\kappa (0) \in B$ define$A (0) = A$
  and, again using the construction from Theorem \ref{VaB}, define $K (0) = V
  (\kappa (0), A)$. Define $\gamma (0) = \Pi_1^{- 1} \Pi_1 (\kappa (0))
  \setminus K (0)$.
  
  Transfinite induction hypothesis: Suppose $i \in J$ and $\gamma : [0, i)
  \rightarrow S$ satisfies, with one possible exception, all of the hypotheses
  of Lemma \ref{induct}, but not necessarily the requirement that $K (j)$ is
  eventually constant.
  
  To be precise suppose $\kappa : [0, i) \rightarrow A \cup B$ is a function.
  Suppose $V (j) = \cup_{k \leq j} \gamma (k)$ is clopen in $G$ for each $j <
  i$. Suppose for each $j < i$ we have $\gamma (j) = (\Pi_{\eta (j)}^{- 1}
  \Pi_{\eta (j)} (\kappa (j))) \setminus K (j)$ with $K (j)$ clopen in $G$ and
  $\kappa (j) < K (j)$. Suppose given $j, \kappa (j)$ and $K (j)$ the index
  $\eta (j)$ is minimal to ensure that $\gamma | [0, j]$ is increasing. Let $U
  (j) = (\Pi_{\eta (j)}^{- 1} \Pi_{\eta (j)} (\kappa (j)))$. We also assume
  for all $j < i$ that $\gamma (j) \cap A = \emptyset$ or $\gamma (j) \cap B =
  \emptyset$. If $j < i$ $W (j) = \cup_{k < j} \gamma (k)$ and we note by
  ${Remark} \ref{wclopen} {that} W (j)$ is clopen.
  
  Suppose $\gamma$ also satisfies the following 2 conditions.
  
  1) Suppose if $j < i$ then $\kappa (j) = min ((A \cup B) \setminus W (j))$,
  if $\kappa (j) \in A$ then $K (j) = V (\kappa (j), B (j))$ with $B (j) = B
  \setminus W (j)$, and if $\kappa (j) \in B$ then $K (j) = V (\kappa (j), A
  (j))$ with $A (j) = A \setminus W (j)$.
  
  2) The sets $V (\kappa (j), B (j))$ and $V (\kappa (j), A (j))$ are defined
  as in Theorem \ref{VaB}.
  
  Now the heart of the matter is understand why $\cup_{j < i} \gamma (j)$ is
  clopen, and for this we argue by contradiction. If $\cup_{j < i} \gamma (j)$
  is not closed apply Lemma \ref{ezinc} and let $c = sup (\cup_{j < i} \gamma
  (j))$. Our plan is to ultimately show there exists $j_A$ so that if $j_A
  \leq j$ then wolog $\kappa (j) \in A$, and $V (\kappa (j), B (j)) = V
  (\kappa (j_A), B (j_A)) = K (j_A)$. It will then follow directly from Lemma
  \ref{induct} that $\cup_{j < i} \gamma (j)$ is clopen, contradicting our
  assumption that $\cup_{j < i} \gamma (j)$ is not closed.
  
  By Lemma \ref{ezinc} obtain an increasing sequence $s_1 < s_2$....terminal
  in $[0, i)$. By Lemma \ref{ezinc} $\kappa (s_n) \rightarrow c$ and hence,
  since $im (\kappa) \subset A \cup B$ and since $A$ and $B$ are disjoint
  closed sets, eventually either $\kappa (s_n) \in A$ or $\kappa (s_n) \in B$.
  Thus wolog $c (s_n) \in A$ eventually. Observe if the increasing sequence
  $t_1 < t_2$.... is also terminal in $[0, i)$ then $\{s_n \}$ and $\{t_n \}$
  are interlaced and hence by Lemma \ref{ezinc} $\kappa (t_n) \rightarrow c$.
  Thus by Lemma \ref{seqcont} the extended function $\kappa | [0, i] = \kappa
  | [0, i) \cup \{i, c\}$ is continuous at $i$ with the order topology on $[0,
  i.]$. In particular there must exist $M$ so that if $s_M \leq j$ then
  $\kappa (j) \in A$, since otherwise we could manufacture a pair of
  interlaced increasing terminal sequences $\{s_n \}$ $\{t_n \}$ with $\kappa
  (s_n) \in A$ and $\kappa (t_n) \in B$ yielding the contradiction $c \in A
  \cap B$.
  
  Next we will show if $s_m \leq j$ then $B_{} (s_m) = B (j)$. Note if $j =
  s_m$ then $B (s_m) = B (j)$. Proceeding by transfinite induction, suppose
  $s_m < j < i$ and $B (s_m) = B (t)$ whenever $s_m \leq t < j$. By definition
  $B (j) = B \setminus W (j)$ and if $s_m \leq t < j$ then $B (t) = B
  \setminus W (t)$ and thus (since $W (t) \subset W (j)$) we have $B (j)
  \subset B (t)$, and in particular $B (j) \subset B (s_m)$. Conversely note
  $W (j) = W (s_m) \cup_{s_m \leq t < j} \gamma (t) = W (s_m) \cup (\cup_{s_m
  \leq t < j}  (U (t) \setminus B_{s_m}))$. By definition $B (s_m) = B
  \setminus W (s_m)$. Thus if $x \in B (s_m)$, then $x \notin W (j)$ and hence
  $x \in B (j) .$
  
  We have established that if $s_m \leq j < i$ then $K (j) = V (\kappa (j), B
  (s_M))$ with $\kappa (j) \in A$. Thus by Theorem \ref{VaB}, for each
  increasing terminal sequence $t_1 < t_2 ... < i$ there exists $M$ so that if
  $M \leq n < m$ then $V (c (t_n), B (t_n)) = V (\kappa (t_m), B (t_m))$. Thus
  if we define $f : [s_m, i) \rightarrow S$ as $f (j) = V (\kappa (j), B (j))$
  then $f| [s_m, j)$ is eventually constant by Lemma \ref{seqcont}. Hence
  there exists $j_A \in [s_m, i)$ so that if $j_A \leq j$ then $\gamma (j) = U
  (j) \setminus V (\kappa (j_A), B (s_M))$. It now follows from Lemma
  \ref{induct} that $\cup_{j < i} \gamma (j)$ is clopen after all.
  
  Now there are two cases.
  
  Case 1. If the the clopen set $W (i) = \cup_{j < i} \gamma (j)$ covers
  neither $A$ nor $B$, define $\kappa (i) =$ min$ ((A \cup B) \backslash W
  (i))$. If $\kappa (i) \in A$ define $K (i) = V (\kappa (i), B\backslash W
  (i$)) as in Theorem \ref{VaB}, define $B (i) = B\backslash W (i) .$If
  $\kappa (i) \in B$ define $K (i) = V (\kappa (i), A\backslash W(i)$) as in Theorem \ref{VaB} and define $A (i) = A\backslash W (i) .$
  
  Now apply Lemma \ref{big} and let $\eta (i)$ be minimal so that $W (i) <
  \Pi_{\eta (i)}$($\kappa (i)$). Define $U (i) = \Pi_{\eta (i)}^{- 1}
  \Pi_{\eta (i)} \kappa (i)$ and define $\gamma (i) = U (i) \backslash K (i)
  .$ Note, by definition, if $\kappa (i) \in A$ and $b \in B$satisfies $b <
  \kappa (i)$ then $b \in W (i)$. If $b \in B$ satisfies $\kappa (i) < b$ then
  $b \in K (i) $ . 
Thus $\gamma (j) \cap B = \emptyset .$ By a symmetric
  argument if $\kappa (i) \in B$ then $\gamma (j) \cap A =X \emptyset$.
  Thus replacing the index $i$ with $i + 1,$and defining $V (i) = W (i) \cup
  \gamma (i),$the transfinite induction hypothesis is preserved. Proceeding
  via transfinite induction we continue the construction in Case 1 until case
  2 is achieved.
  
  Case 2. If the clopen set $W (i) = \cup_{j < i} \gamma (j)$ covers $A$ or
  $B$ then wolog $A \subset W (i)$. Let $T_A = \{j < i| \kappa (j) \in A\}$.
  Let $U (A) = \cup_{j \in T_A} \gamma (j)$. We must show $U (A)$ is a clopen
  set such that $A \subset U (A)$ and $B \cap U (A) = \emptyset$, it will then
  follow that $U (A)$ and $G \setminus U (A)$ are disjoint clopen sets
  covering $A$ and $B$ respectively.
  
  That $U (A)$ is clopen follows from basic topology. Given any collection $R$
  of pairwise disjoint clopen sets in the space $X$, if the union is clopen in
  $X$, then the union taken over any subset $H \subset R$ will also be clopen
  in $X$. By hypothesis if $j < i$ then $j \in T_A$ iff $\gamma (j) \cap B =
  \emptyset .$ Thus $U (A)$ is a clopen set covering $A$ such that $U (A) \cap
  B = \emptyset .$
\end{proof}

\section{Applications}\label{apps}

The motivation for this paper is to prove (with the quotient topology
inherited from the space of based loops), that the Hawaiian earring group has
large inductive dimension zero. Starting from a combination of first
principles, the basics of finite free groups, and the knowledge that
${HE}$ is $\pi_1$ shape injective, we will ultimately reduce the question
of calculating the dimension of $\pi_1 ({HE}, p)$ to Theorem
\ref{mainapp}.

\

\begin{thm}
  \label{mainapp}Suppose for each $n \in \{ 1, 2, 3, \ldots . \}$, $X_n $is a
  discrete space. We assume $X_1 \subset X_2 \ldots .$ and for each $n$ the
  map $R_n : X_{n + 1} \rightarrow X_n$ is a retraction. Let the space
  $X_{\infty} = \lim_{\leftarrow} X_n$ \ denote the topological inverse limit,
  i.e. $X_{\infty}$ is the subspace of the countable product $X_1 \times X_2
  \times \ldots$ so that $(x_1, x_2, \ldots) \in X_{\infty}$ iff $R_n (x_{n +
  1}) = x_n$ for each $n.$
  
  Suppose furthermore we have a sequence of quotient maps $q_n : X_n
  \rightarrow G_n $ so that the formula $q_n R_n q_{n + 1}^{- 1} = r_n$
  induces a map such that $r_n q_{n + 1} = q_n R_n$, i.e. there is an induced
  map $r_n : G_{n + 1} \rightarrow G_n$, commuting with the retraction $X_{n +
  1} \rightarrow X_n .$Finally suppose the formula $q_{n + 1} q_n^{- 1}
  $induces an embedding $j_n : G_n \rightarrow G_{n + 1}$, commuting with
  inclusion $X_n \rightarrow X_{n + 1} .$
  
  By definition the maps $\{ q_n \}$ induce an equivalence relation on
  $X_{\infty}$ such that $(x_1, x_2, \ldots) \sim (y_1, y_2, \ldots)$ iff $q_n
  (x_n) = q_n (y_n)$ for each $n.$ Define $G$ as the corresponding topological
  quotient $q : X_{\infty} \rightarrow G$. Then $G$ has large inductive
  dimension 0.
\end{thm}

\begin{proof}
  Our strategy is to apply Theorem \ref{gzero} by first showing $G$ can be
  made to satisfy axioms \ref{retracts}, \ref{lin} and \ref{trap}. For axioms
  \ref{lin} and \ref{trap}, we must build a linear order on $X_{\infty}$ which
  induces a suitable linear order on $G$. This is ultimately straightforward,
  but with a few restrictions imposed by the starting data $\{ q_n \},$ the
  need to ensure axiom \ref{trap}, and the need to ensure that $R_n \leq
  {id} |X_{n + 1} .$ To define a lexical order on $X_{\infty}$ it
  suffices to proceed recursively, first defining a well ordering of $X_1,$
  then extending to a well ordering of $X_2$ and so on.
  
  To impose a well ordering on $X_1 $, first arbitrarily well order $G_1$,
  and \ then arbitrarily well order each point primage under $q_1 : X_1
  \rightarrow G_1$.
  
  Now we define on $X_1$ a kind of local lexical order as follows. To compare
  two points $\{ x_1, y_1 \} \subset$ $X_1,$ if $q_1 (x_1) \neq q_1 (y_1)$ let
  the order in $G_1$ determine which is bigger. If $q_1 (x_1) = q_1 (y_1)$ let
  the order on point preimages of $q_1$ decide which is bigger. Crucially if
  $q_1 (y_1) \neq q_1 (x_1)$ and $q_1 (x_1) = q_1 (z_1)$ then $y_1 < \{ x_1,
  z_1 \}$ or $y_1 > \{ x_1, z_1 \}$.
  
  To extend the well ordering of $X_1$ to $X_2$ we begin as follows.
  First, for each $x_1 \in X_1 $define $X_2 (x_1) = X_2 \cap R_1^{- 1}
  (x_1),$and note $x_1 \in X_2 (x_1)$ since $X_1 \subset X_2$ and $R_1 (x_1) =
  x_1 .$ Next, define $G_2 (x_1) = q_2 (X_2 (x_1))$ and note $G_2 (x_1)
  \subset G_2 .$ Now well order $G_2 (x_1)$ to have minimal element $q_2
  (x_1)$ and otherwise the well ordering of $_{} G_2 (x_1)$ is arbitrary.
  Next, well order each point preimage of the map $q_2 |X_2 (x_1) : X_2 (x_1)
  \rightarrow G_2 (x_1)$ subject only to the constraint that $x_1 = \min q_2
  (x_1) .$ Thus $x_1 = \min \{ X_2 (x_1) \}$.
  
  To complete the definition of the well ordering on $X_2 $ suppose we
  are,given distinct points $\{ x_2, y_2 \} \subset X_2$. If
  $R_1 (x_2) \neq R_1 (y_2)$ we require the order of $X_1$ to dictate which is
  bigger. If $R_1 (x_2) = R_1 (y_2)$ and $q_2 (x_2) \neq q_2 (y_2)$ we require
  the order of $G_2 (R_1 (x_2))$ to dictate which is bigger. If $R_1 (x_2) =
  R_1 (y_2)$ and $q_2 (x_2) = q_2 (y_2)$ we require the order on $q_{} (x_2)$
  to dictate which is bigger. In summary, lexical inspection of the ordered
  triples $(R_1 (x_2), q_2 (x_2), x_2)$ and $(R_1 (y_2), q_2 (y_2), y_2)$
  determines which of $\{ x_2, y_2 \}$ is bigger.
  
  Crucially, if $\{ x_2, y_2, z_2 \} \subset X_2$ and $x_2 < y_2 < z_2$ and
  $q_1 R_1 (x_2) = q_1 R_1 (z_2),$ then $q_1 R_1 (x_2) = q_1 R_1 (y_2)$, and
  we argue the contrapositive as follows. Suppose $q_1 R_1 (x_2) \neq q_1 R_1
  (y_2)$ and $q_1 R_1 (x_2) = q_1 R_1 (z_2)$ with $\{ x_2, y_2, z_2 \} \subset
  X_2$. Let $x_1 = q_1 (x_2)$ and $y_1 = q_1 (y_2)$ and $z_1 = q_1 (z_2)$. As
  noted we must have $y_1 < \{ x_1, z_1 \}$ or $y_1 > \{ x_1, z_1 \}$, and
  hence by definition $y_2 < \{ x_2, z_2 \}$ or $y_2 > \frac{}{} \{ x_2, z_2
  \}$. Finally note $R_1 \leq {id} |X_2 .$
  
  Proceeding recursively, suppose $X_{n - 1}$ has been well ordered so that if
  $q_{n - 1} (x_{n - 1}) \neq q_{n - 1} (y_{n - 1}) $ and $q_{n - 1} (x_{n -
  1}) = q_{n - 1} (z_{n - 1})$ ${then} y_{n - 1} < \{ x_{n - 1}, z_{n -
  1} \}$, or $y_{n - 1} > \{ x_{n - 1}, z_{n - 1} \}$. Suppose $R_{n - 2}
  \leq {id} |X_{n - 1} .$
  
  For each $x_{n - 1} \in X_{n - 1}$ define $X_n (x_{n - 1}) = X_n \cap R_{n -
  1}^{- 1} (x_{n - 1})$, and note $x_{n - 1} \in X_n (x_{n - 1})$ since $X_{n
  - 1} \subset X_n$ and $R_{n - 1} (x_{n - 1}) = x_{n - 1} .$ Next, define
  $G_n (x_{n - 1}) = q_n (X_n (x_{n - 1}))$ and note $G_n (x_{n - 1}) \subset
  G_n .$ Now well order $G_n (x_{n - 1})$ to have minimal element $q_n (x_{n -
  1})$ and otherwise the well ordering of $_{} G_n (x_{n - 1})$ is arbitrary.
  Next, well order each point preimage of the map $q_n |X_n (x_{n - 1}) : X_n
  (x_{n - 1}) \rightarrow G_n (x_{n - 1})$ subject only to the constraint that
  $x_{n - 1} = \min q_n (x_{n - 1}) .$
  
  To complete the definition of the well ordering on $X_n $ suppose we
  are,given ${distinct} {points} \{ x_n, y_n \} \subset X_n$. If
  $R_{_{} n - 1} (x_n) \neq R_{n - 1} (y_n)$ we require the order of $X_{n -
  1}$ to dictate which is bigger. If $R_{n - 1} (x_n) = R_{n - 1} (y_n)$ and
  $q_n (x_n) \neq q_n (y_n)$ we require the order of $G_n (R_{n - 1} (x_n))$
  to dictate which is bigger. If $R_{n - 1} (x_n) = R_{n - 1} (y_n)$ and $q_n
  (x_n) = q_n (y_n)$ we require the order on $q_{} (x_n)$ to dictate which is
  bigger. In summary, lexical inspection of the ordered triples $(R_{n - 1}
  (x_n), q_n (x_n), x_n)$ and $(R_{n - 1} (y_n), q_n (y_n), y_n)$ determines
  which of $\{ x_n, y_n \}$ is bigger.
  
  Crucially, if $\{ x_n, y_n, z_n \} \subset X_n$ and $x_n < y_n < z_n$ and
  $q_{n - 1} R_{n - 1} (x_n) = q_{n - 1} R_{n - 1} (z_n),$ then $q_{n - 1}
  R_{n - 1} (x_n) = q_{n - 1} R_{n - 1} (y_n)$, and we argue the
  contrapositive as follows. Suppose $q_{n - 1} R_{n - 1} (x_n) \neq q_{n - 1}
  R_{n - 1} (y_n)$ and $q_{n - 1} R_{n - 1} (x_n) = q_{n - 1} R_{n - 1} (z_n)$
  with $\{ x_n, y_n, z_n \} \subset X_n$. Let $x_{n - 1} = q_{n - 1} (x_n)$
  and $y_{n - 1} = q_{n - 1} (y_n)$ and $z_{n - 1} = q_{n - 1} (z_n)$. By the
  induction hypothesis we must have $y_{n - 1} < \{ x_{n - 1}, z_{n - 1} \}$
  or $y_{n - 1} > \{ x_{n - 1}, z_{n - 1} \}$, and hence, appealing to our
  local definition, $y_n < \{ x_n, z_n \}$ or $y_n > \frac{}{} \{ x_n, z_n
  \}$. Finally note $R_{n - 1} \leq {id} |X_n .$
  
  To check that axiom \ref{retracts} holds, note the map $X_n \rightarrow
  X_{\infty}$ sending $x_n$ to $(x_1, \ldots, x_n, x_n, \ldots)$ is \ a
  topological embedding, henceforth we conflate the discrete space
  $X_n$ with the corresponding discrete subspace of eventually
  constant sequences in$ X_{\infty}$, the sequences whose
  terms coincide from index $n$ upward. Thus, with moderate abuse of notation,
  we extend the map $R_N |X_{N + 1}$ canonically to $R_N : X_{\infty}
  \rightarrow X_N$, so that $R_N (x_{1,} x_2, \ldots x_N, x_{N + 1}, \ldots .)
  = (x_{1,} x_2, \ldots x_N, x_N, \ldots)$, and note $R_N = R_N R_{N + 1} .$
  
  By definition a point $g \in G$ is a subspace $g \subset X_{\infty}$ so
  that distinct points $\{ x, y \} \subset g$ have $q_n$ equivalent
  coordinates for each $n$,  and a point $g_n \in G_n$ is a subspace $g_n \in
  X_n$. Our starting assumptions ensure we have a well defined function $G_n
  \hookrightarrow G$ sending $g_n \in G_n$ to $(\ldots .r_{n - 1} (g_n), g_n,
  j_n (g_n)_{}, j_{n + 1} j_n (g_n) \ldots) \in G$. This is a topological
  embedding, and henceforth we conflate $G_n$ with the corresponding subspace
  of $G$. The map $q_n R_n : X_{\infty} \rightarrow G_n$ is constant on sets
  of the form $q^{- 1} (g)$ and thus there is a unique induced map $\Pi_n : G
  \rightarrow G_n$ such that $\Pi_n = \Pi_n \Pi_{n + 1} .$ The maps $\{ \Pi_n
  \}$ determine a continuous injection $\phi : G \rightarrow \lim_{\leftarrow}
  G_n$.
  
  If $\phi$ were a topological embedding then it would follow easily that
  ${dimG} = 0$, since $\phi$ embeds $G$into the 0 dimensional metric
  space $G_1 \times G_2 \ldots$However in general $\phi$ is NOT a topological
  embedding {\cite{FA0}}. \
  
  Note each space $G_n$ is discrete, since topological quotients of discrete
  spaces are discrete. Moreover $G$is a quotient of a metrizable space and
  hence $G$is sequential. Thus axiom \ref{retracts} will hold provided we can
  show pointwise convergence $\Pi_n \rightarrow {id} |G.$ The latter
  claim holds by definition, shown as follows. Given $g \in G$ and an open $U
  \subset G$ so that $g \in U$, lift $g$ to some $(x_1, x_2, \ldots) \in q^{-
  1} (g) \subset X_{\infty}$ and obtain a basic open $V \subset q^{- 1} (U)
  \subset X_{\infty}$ with $(x_1, x_2, \ldots) \in V$ so that $V = \{ x_1 \}
  \times \{ x_2 \} \ldots \times \{ x_N \} \times X_{N + 1} \ldots .$Now given
  $n \geq N$ to check $\Pi_n (g) \in U$ it suffices to check, since $q :
  X_{\infty} \rightarrow G$ is a quotient map, that some lift of $\Pi_n (g)
  \in U$. Our special point $(x_1, x_2, \ldots) \in V$ suffices. Thus axiom
  \ref{retracts} holds.
  
  The space $\lim_{\leftarrow} X_n = X_{\infty}$ is a topological inverse
  limit of discrete spaces under retraction bonding maps $R_n |X_{n + 1} .$ We
  have well ordered each set $X_n $ so that $R_n (x_{n + 1}) \leq x_n
  $for $x_{n + 1} \in X_{n + 1} .$ Thus with the induced lexigraphic order on
  $X_{\infty}$ we have $R_n (x) \leq x$for $x \in X_{\infty} .$
  
  The space $X_{\infty} $ has the topology of pointwise convergence, and
  thus, since $X_n$ is a discrete space, a sequence $\{ s_n \} \subset
  X_{\infty}$ converges iff $\Pi_N (s_n)$ is eventually constant for each $N$.
  Thus, if the strictly increasing sequence $s_1 < s_2 \ldots \subset
  X_{\infty}$\quad diverges, then every subsequence of $\{ s_n \}$ diverges,
  since $\{ \Pi_N (s_n) \}$ is nondecreasing for each $N.$ Conversely, if the
  increasing sequence $\{ s_n \} \subset X_{\infty}$ converges, then, since
  $\{ \Pi_N (s_n) \}$ is nondecreasing and eventually constant for each $N,$
  $\lim \{ s_n \} = \sup \{ s_n \}$.
  
  Since each set $X_n$ is well ordered, and since no well ordered set admits
  a strictly decreasing sequence with infinitely many terms, each strictly
  decreasing sequence $s_1 > s_2 \ldots . \subset X_{\infty}$ converges.
  Consequently if $B \subset X_{\infty}$ is nonempty and closed, then $\min
  (B)$ exists. To see why, let $b_1 = \min \Pi_1 (B)$. Let $b_2 = \min \Pi_2
  \Pi_1^{- 1} (b_1) .$ Let $b_3 = \min \Pi_3 \Pi_2^{- 1} (b_2)$ and so on.
  Note $b_1 \geq b_2 \ldots$ and let $b = \lim \{ b_n \}$. Thus $b \in B$
  since $B$ is closed .Note $b \leq a$ for each $a \in B$
  and thus $b = \min (B)$.
  
  Note $G$ is $T_2 $ by Remark \ref{gt2} and in particular $G$ is $T_1$ .
  Thus point preimages are closed under the map $q : X_{\infty} \rightarrow
  G$. Define $\sigma : G \rightarrow X_{\infty}$ as $\sigma (g) = \min (q^{-
  1} (g))$. Since $\sigma$is one to one and since subsets of linearly ordered
  sets are linearly ordered, we obtain a linear order on $G$ defined so that
  $g < h$ iff \ $\sigma (g) < \sigma (h) .$ In particular the sequence $\{ g_n
  \}$ is strictly increasing in $G$ iff $\{ \sigma (g_n) \}$ is strictly
  increasing in $X_{\infty}$. In the process of checking axiom \ref{lin} we
  will show $\sigma$ is left continuous but not right continuous.
  
  First we observe some basic properties of our definition of the linear order
  on the set $X_{\infty}$, conflating $X_N$ with the subspace of $X_{\infty}$,
  all of whose terms are constant from index $N$ and above. If $x = (x_1, x_2,
  \ldots .) \in X_{\infty}$ then $x_n \leq x_{n + 1} \leq x,$and
  $x_n \rightarrow x,$ and $x = \sup \{ x_n \} .$
  
  Next we observe a basic property of $\sigma$. If $(y_1, y_2, \ldots .) = y
  < z = \sigma (g) = (z_1, z_2, \ldots) \in X_{\infty}$ and if $N$ is minimal
  so that $y_N \neq z_N$, then $y_N < z_N$, $\min (q_N (y_N)) < z_N = \min
  (q_N (z_N))$ and hence $q_N (y_N) < q_N  (z_N) .$ \
  
  To check axiom \ref{lin}, if $A \subset G$ is closed, then $q^{- 1} (A) =
  B$ is closed in $X_{\infty}$ and thus ${minA} = q ({minB})$ and in
  particular $A$ has a minimal point. Suppose $\{ g_n \}$ is a strictly
  increasing sequence in $G$.
  
  Case 1. Suppose the corresponding increasing sequence $\{ \sigma (g_n) \}$
  converges to some $x = (x_{1,} x_2, \ldots) \in X_{\infty}$. Note the
  sequences $\{ \sigma (g_n) \}$ and $\{ x_n \}$ are interlaced.
  
  Sequential continuity of $q$ at $x$ shows $\{ g_n \}$ converges to some $q
  (x) = g \in$ $G$. We will show $x = \sigma (g)$ and $g = \sup \{ g_n \} .$
  
  Suppose $y = (y_1, y_2, \ldots) \in X_{\infty}$ with $y < x$. Obtain $N$
  minimal so that $y_N < x_N .$ Obtain $M$ minimal so that $R_N (\sigma (g_M))
  = x_N .$ Note $R_n (y) = R_n (x) = R_n (\sigma (g_M))$ if $n < N$. Thus $q_N
  (y_N) < = q_N (\sigma (g_M)) = q_N (x_N)$ and hence $q (y) \neq q (x)$. This
  shows $x = \min (q (x))$, i.e. $x = \sigma (h)$ for some $h \in G$. However,
  by definition $q \sigma = {id} |G$, and thus $g = q (x) = q \sigma (h)
  = h$. Hence $x = \sigma (g)$. This argument also shows if $\sigma (k) = y <
  x = \sigma (g) $then$ y < \sigma (g_M)$ and hence $g = \sup \{ g_n \}
  .$
  
  Case 2. If the increasing sequence $\{ \sigma (g_n) \}$ divereges in
  $X_{\infty}$ then every subsequence of $\{ \sigma (g_n) \} $ diverges in
  $X_{\infty}$ and hence, since $q$ is a quotient map, every subsequence of
  $\{ g_n \}$ diverges in $G$.
  
  The injection $\sigma |G_n$ determines that $G_n$ is well ordered. Now we
  must check that $\Pi_n \leq \Pi_{n + 1} \leq {id} |G$ while
  keeping in mind that $G$ does NOT in general have the lexigraphic order.
  That is to say, it can happen that $g < h$ in $G$ but $\Pi_n (g) \geq
  \Pi_n (g)$ with $n$ minimal so that \ $\Pi_n (g) \neq \Pi_n (g)$.
  
  Given $g \in G$ let $x = \sigma (g) \in X_{\infty}$ and let $x_n = R_n (x)
  = R_n (R_{n + 1} (x)) = R_n (x_{n + 1})$ with $x_{n + 1} = R_{n + 1} (x) .$
  Our recursive definition of the linear order on $X_{n + 1}$ ensures $\min
  (q_n (x_n)) \leq \min (q_{n + 1} (x_{n + 1}))$ and inclusion $G_n
  \rightarrow G$ ensures $\min (q_{n + 1} (x_{n + 1})) \leq \min (q (x))$
  and hence \ $\Pi_n \leq \Pi_{n + 1} \leq {id} |G$. Thus axiom
  \ref{lin} holds.
  
  To check axiom \ref{trap} we first decode the notion of blowups in the
  context of $X_{\infty}$. By definition $G_{\infty} = G_1 \cup G_2 \ldots$
  and $\sigma (G_{\infty}) \subset X_{\infty}$. Given $x \in \sigma
  (G_{\infty})$ let $x = \sigma (g)$, and obtain $N$ minimal so that $g_{} \in
  G_N$. Thus $x \in X_N$ with $R_{N - 1} (x) < R_N (x) = R_{N + k} (x) .$ Thus
  $x_N \in X_N \backslash X_{N - 1}$ and hence by definition of our linear
  order on $X_N$ we have $x = \min (q_N (x))$. By definition ${blowup}
  (g) = \Pi_N^{- 1} \Pi_N (g) = \Pi_N^{- 1} (g)_{} \subset G.$ The pullback in
  $X_{\infty}$ is precisely $q^{- 1} (g)$.
  
  The lexigraphic order on $X_{\infty}$ ensures the following lifted version
  of axiom \ref{trap} holds. Suppose $x_1 < x_2 < x_3 $ with $x_n \in X_{N_n}$
  and $N_n$ minimal so that the mentioned membership holds, but also so that
  $R_{N_3}^{^{- 1}} (x_3) \subset R_{N_1}^{^{- 1}} (x_1)$. Then $R_{N_2}^{^{-
  1}} (x_2) \subset R_{N_2}^{^{- 1}} (x_2)$. To see why, note $R_{N_1} (x_3) =
  x_1$ and hence by definition of lexigraphic order we have $R_{N_1} (x_2) =
  x_1$.
  
  Thus given points $k_1 < k_2 < k_3$ in $G_{\infty}$ so that ${blowup}
  (k_3) \subset {blowup} (k_1)$, let $x_n = \sigma (k_n)$ and deduce
  ${blowup} (k_2) \subset {blowup} (k_1)$. Hence axiom \ref{trap}
  holds.
\end{proof}

\

\begin{corollary}
  \label{he}The Hawaiian earing ${HE}$is a subpace of the plane, the
  union of a sequence circles centered at $(0, 1 / n)$ with radius $1 / n$ and
  sharing the common point $(0, 0)$. The Hawiian earring group $G$ is the
  fundamental group of ${HE}$, the set of path components of the spaced
  of based loops in ${HE}$, with group operation cancatanation. Endowed
  with the quotient topology inherited from the space of based loops in
  ${HE}$, the Hawaiian earring group has large inductive dimension 0.
\end{corollary}
  
  \begin{proof}
    Let $L (S^1, 1, {HE}, p)$ denote the space of based loops in
    ${HE} .$ Note $L (S^1, 1, {HE}, p)$ is a separable metric space
    with the uniform metric topology ( equivalent in this case to the compact
    open topology), since both $S^1$ and ${HE}$ are compact metric
    spaces. Our plan is to manufacture a space $G$ as in the hypothesis of
    Theorem \ref{mainapp}, and build a quotient map $F : L (S^1, 1, {HE},
    p) \rightarrow G$ whose point preimages are precisely the path components
    of $L (S^1, 1, {HE}, p) .$ Consequently, by basic general topology,
    there is an induced homeomorphism $h : \pi_1 ({HE}, p) \rightarrow
    G$, and hence both spaces have the same dimension.
    
    If $p \in {HE}$ is the interesting point, given an aribtrary based
    loop $f \in L (S^1, 1, {HE}, p)$, for each component $J \subset f^{-
    1}  ({HE} \backslash p)$, we can homotopically tighten $f| \bar{J}$
    within its image to a linearly parameterized loop or to a constant. Since
    $f$ is uniformly continous, the union of the tightenings $f_{{wt}}$
    is continuous and path homotopic to $f$, and we call the resulting map
    $f_{{wt}}$ weak tight. Thus, a loop $f_{{wt}} \in L (S^1, 1,
    {HE}, p)$ is weak tight provided $f|J$ is linear and one to one, for
    each component $J \subset f^{- 1}_{{wt}}  ({HE} \backslash p)$.
    Notice the weak tight loops ${WT} (S^1, 1, {HE}, p)$ comprise a
    closed subspace of $L (S^1, 1, {HE}, p)$, since being not weak tight
    is an open property for loops in $L (S^1, 1, {HE}, p) .$
    
    If H denotes the group of orientation preserving homeomorphisms of $S^1$
    which fix $1$, then $H$acts isometrically on ${WT} (S^1, 1,
    {HE}, p)$ via right composition, i.e. $h \in H$ sends $f_w \in
    {WT} (S^1, 1, {HE}, p)$ to the map $f_w h.$ Thus by Lemma
    \ref{quotient} the quotient ${WT} (S^1, 1, {HE}, p) /
    \overline{H ({WT} (S^1, 1, {HE}, p))}$ is metrizable. For
    convenience rename the mentioned quotient space $X_{\infty} $and the
    quotient map $Q_w : {WT} (S^1, 1, {HE}, p) \rightarrow
    X_{\infty}$. A typical point of $X_{\infty}$ is a weak tight path up to
    monotone orientation preserving reparameterization, i.e. two weak tight
    paths are equivalent if they pass through the same points in the same
    order.
    
    Let ${HE}_n \subset {HE}$ denote the bouquet of the first $n$
    loops. Thus if we define $X_n \subset X_{\infty}$ as the subspace with
    image in ${HE}_n$ we have an induced retraction $R_n : X_{\infty}
    \rightarrow X_n$ deleting all large index loops. Crucially notice $X_n$ is
    the discrete monoid on $n$ letters $\{ x_1, x_1^{- 1}, \ldots x_n^{- 1}
    \}$, with the empty word corresponding to the constant loop at $p$, and
    $X_{\infty} = \lim_{\leftarrow} X_n $with bonding map $R_n |X_{n + 1} .$
    Thus we can think of points of $X_{\infty}$ as unreduced infinite words in
    $\{ x_1, x_1^{- 1}, x_{2, \ldots} \}$ so that each letter appears finitely
    many times.
    
    Now let $G_n $ denote with the discrete topology, the free group on $n$
    letters $\{ x_1, \ldots x_n \}$ and let $q_n : X_n \rightarrow G_n$ denote
    the canonical quotient map. Note the maps $\{ q_n \} $induce an
    equivalence relation on $X_{\infty}$: two infinite words $w \in
    X_{\infty}$ and $v \in X_{\infty}$ are equivalent iff for all $n$ $q_n R_n
    (w)$=$q_n R_n (w) \in G_n .$ \ Let $q : X_{\infty} \rightarrow G$ denote
    the corresponding quotient map determined by this equivalence relation.
    
    The previous paragraphs establish a composition of functions $L (S^1, 1,
    {HE}, p) \rightarrow {WT} (S^1, 1, {HE}, p) \rightarrow
    X_{\infty} \rightarrow G$. The first arrow is a discontinous retraction,
    the second and third arrows are continuous quotient maps, and we let F
    denote the composition $L (S^1, 1, {HE}, p) \rightarrow G.$ By
    definition $\pi_1 ({HE}, p)$ is the quotient of $L (S^1, 1,
    {HE}, p)$ modding out by the path components. Thus, to prove the
    existence of an induced homeomorphism $h : \pi_1 ({HE}, p)
    \rightarrow G$, it suffices, by basic general topology, to show that $F$
    is a quotient map whose point preimages are the path components of $L
    (S^1, 1, {HE}, p) .$ Let $W : L (S_1, 1, {HE}, p) \rightarrow
    {WT} (S^1, 1, {HE}, p)$ denote the discontinuous retraction
    described previously.
    
    To check continuity of $F$ suppose $f_n \rightarrow f$ uniformly in $L
    (S^1, 1, {HE}, p)$. We apply Lemma \ref{convcrit} to the sequence $\{
    F (f_n) \}$ and first show $\{ \Pi_N (F (f_n)) \}$ is eventually constant
    for each $N$. Let $\kappa_N : {HE} \rightarrow {HE}_N$ denote
    the canonical retraction, notice locally at $f,$ the composition $\{
    \kappa_N (f_n) \}$ eventually preserves the homotopy path class of $\{
    \kappa_N (f) \}$ in ${HE}_N .$ Thus $\{ \Pi_N (F (f_n)) \}$ is
    eventually constant. To check that $\{ \sigma (F (f_n)) \}$ has compact
    closure we will apply Ascoli's Theorem. First note the map $W$ preserves
    or improves \ equicontinity data (and the image of $1 \in S^1$ is constant
    and thus convergent), and hence $\{ W (f_n) \}$ has compact closure in
    ${WT} (S^1, 1, {HE}, p)$.
    
    The following definition has an algebraic analogue, the different ways
    that one might start with an unreduced word in $X_N$ and then cancel
    inverse pairs to create the irreducible representive. Given a weak tight
    loop $\beta \in {WT} (S^1, 1, {HE}, p)$ and a natural number
    $N$, define $\Sigma (\beta, N) \subset {WT} (S^1, 1, {HE}, p)$
    as the subspace of irreducicle loops in ${HE}_N$, obtained by
    starting with $\kappa_N (\beta)$ and deleting successive nonconstant $p$
    based inessential loops, replacing each with the constant map $p.$
    
    The important observation is that each loop in $\Sigma (\beta, N)$ has
    equicontinuity data no worse than that of $\beta .$ Thus, since $\{ W
    (f_n) \}$ has compact closure, the union over $N$ and $n$of the subspaces
    $\Sigma (W (f_n), N)$ has compact closure in ${WT} (S^1, 1,
    {HE}, p)$. Call the latter compactum $C$, recall Lemma \ref{convcrit}
    and observe $\{ \sigma F (f_n) \} \subset C$. Thus $F$ is continous by
    Lemma \ref{convcrit}. Since ${qQ}_w$ is a quotient map, and since the
    (discontinous) map $W$ is a retraction, it follows from Lemma
    \ref{ezretract} that $F$ is a quotient map.
    
    To see that point preimages of $F$ are precisely the path components of $L
    (S^1, 1, {HE}, p),$ first note $W (\beta)$ is path homotopic in
    ${HE}$ to $\beta$. Thus if $\alpha$ and $\beta$ are path homotopic in
    HE then $W (\alpha)$ and $W (\beta)$ are path homotopic in ${HE}$.
    Thus $q_n R_n (Q_w W (\alpha)) = q_n R_n (Q_w W (\beta))$ for all $n$ and
    hence $F (\alpha) = F (\beta) .$ Conversely, since the Hawaiian earring is
    $\pi_1$ shape injective, if $\alpha$ and $\beta$ are not path homtopic in
    ${HE}$ then $q_n R_n (Q_w W (\alpha)) \neq q_n R_n (Q_w W (\beta))$
    ofr some $n$ and hence $F (\alpha) \neq F (\beta) .$ Thus $\pi_1
    ({HE}, p)$ with the quotient topology, is homeomorphic to $G$. It
    follows from Theorem \ref{mainapp} that $G$ has large inductive dimension
    zero, and hence $\pi_1 ({HE}, p)$ has large inductive dimension zero.

  \end{proof}

\section{Miscellaneous}

\begin{lem}
  \label{quotient}Suppose $(X, d)$ is a metric space and $H$is a group (under
  function composition) of isometries of $X .$ Then the orbit closures under
  the action forms a partition of $X$ and the Hausdorff metric is compatible
  with the quotient topology. (It is not necessary to assume $X$ is complete,
  that the orbits are bounded, or that the action is free.) Given $x \in
  X$define $C (x) = \overline{\{ H (x) \}_{} .}$ Thus $C (x)$ is a typical
  element of the quotient space. With moderate abuse of notation we denote the
  quotient space $X / \overline{H (X)}$).
\end{lem}
  
  \begin{proof}
    Given $x \in X$define $C (x) = \overline{\{ H (x) \}_{} .}$ To check the
    orbit closures are disjoint, given $y \in C (x)$ let $y = {limh}_n
    (x)$ for some sequence $\{ h_n \} \subset H$. Suppose $\epsilon > 0$ and
    $z \in H (y)$. Let $z = h (y)$, get $n$ so that $d (y, h_n (x)) <
    \epsilon$. Then $d (z, {hh}_n (x)) < \epsilon$. Thus $z \in C (x)$
    and hence $H (y) \subset C (x)$. Thus $C (y) \subset C (x) $since $C (x)$
    is closed. By a symmetric argument $C (x) \subset C (y)$ and thus $C (x) =
    C (y) .$ Thus the sets of the form $C (x)$ determine a partition of $X$
    into pairwise disjoint closed sets.
    
    Given orbit closures $C (x)$ and $C (y)$ let $\varepsilon$ denote inf $\{
    d (x, y) \}$ \ taken over all $z \in H (y)$. Define $D (C (x), C (y)) =
    \varepsilon$. It is straight forward to check this is the Hausdorff
    metric, and the canonical map $X \rightarrow X / \overline{H (X)}$ is a
    contraction. To check it's a quotient map. Suppose $A \subset X /
    \overline{H (X)}$ is not closed with $C (a_n) \rightarrow C (x)$ with $C
    (a_n) \in A$ and $C (x) \notin A$. Obtain $x_n \in C (a_n)$ with $x_n
    \rightarrow x$. Thus the preimage of $A$ is not closed in $X$.
    
    \
    
    \ 
  \end{proof}

\begin{lem}
  \label{ezretract}Suppose $X$ is a space and $r : X \rightarrow Y$ is a
  (possibly discontinous) retraction onto the subspace $Y$. Suppose $q : Y
  \rightarrow Z$ is a (continuous ) quotient map and $A \subset Z$ is not
  closed. Then $({qr})^{- 1} (A)$ is not closed in X.
\end{lem}

\begin{proof}
  Pullback $A$ to $B = q^{- 1} (A) \subset Y.$ Note $B$ is not closed in $Y$
  since $q$ is a quotient map. Thus $B = ({qr})^{- 1} (A)$ is not closed
  in $X.$
\end{proof}

\begin{lem}
  \label{inject}Suppose the $p$ based loops $f_n$ and $g_n$ are path homotopic
  in the Hawaiian earring ${HE}$with $p$ the special point. Suppose $f_n
  \rightarrow f$ uniformly and $g_n \rightarrow g$ uniformly. Then $f$ and $g$
  are path homotopic.
  \end{lem}
  \begin{proof}
    Cancatanating with the reverse path, the inessential loops $f_n g_n^{- 1}
    \rightarrow {fg}^{^{- 1}}$ uniformly. Since the bouquet of $n$ loops
    ${HE}_n$ is locally contractible, ${fg}^{- 1}$ is inessential in
    ${HE}_n$ for each $n$ and hence, since ${HE}$ is $\pi_1$ shape
    injective ${fg}^{- 1}$ is inessential. See also a direct proof in
    {\cite{DeSmit}}

  \end{proof}

\begin{lem}
  \label{convcrit}Suppose $X_n$ is the discrete free monoid on letters $\{
  x_1, x_1^{- 1}, x_2, \ldots .x_n^{- 1} \}$ with empty identity. Suppose $R_n
  |X_{n + 1} \rightarrow X_n$ is the forgetful retraction, deleting all
  occurences of $\{ x_n, x_n^{- 1} \}$. Let $X_{\infty} = \lim_{\leftarrow}
  X_n$ . Identify $X_n$ with the subspace of $X_{\infty}$ comprised of all
  sequences eventually constant from index $n$ onward. Let $R_n : X_{\infty}
  \rightarrow X_n$ denote the canonical retraction. Let $q_n : X_n \rightarrow
  G_n$ denote the canonical quotient onto the free group $G_n$ on $n$ letters.
  Let $\sigma_n : G_n \rightarrow X_n$ denote the embedding mapping $g_n \in
  G_n$ to its maximally reduced representative. Let $q : X_{\infty}
  \rightarrow G$ denote the quotient map under the equivalence relation $u
  \sim v$ iff $q_n R_n (u) = q_n R_n (v)$ for all $n.$ By definition $q_n R_n$
  descends to the quotient inducing a map $\Pi_n : G \rightarrow G_n .$ \
  Claim 1. There is a well defined (discontinuous) injection $\sigma : G
  \rightarrow X_{\infty}$ with $\sigma (g) = \lim_{n \rightarrow \infty}$
  $\sigma_n \Pi_n (g)$ and $\sigma (g) \in g.$ Claim 2. The sequence $\{ g_n
  \}$ converges in the space $G$ iff for all $N$ the sequence $\Pi_N (g_n) $
  is eventually constant, and also if the sequence $\{ \sigma (g_n) \}$ has
  compact closure.
\end{lem}

\begin{proof}
  We have a canonical partial order on $X_{\infty}$ defined as follows. Given
  $w = (w_1, w_2, \ldots .) \in X_{\infty}$ let $T (w_1, w_2, \ldots) = (N_1
  (w), N_2 (w), \ldots)$ with $N_k (w) \geq 0$ the combined number of
  occurences of $\{ x_k, x_k^{- 1} \}$ in $w_k .$ The function $T$ determines
  a partial lexigraphic order on $X_{\infty}$ with $T (v) < T (w)$ if $N_k (v)
  < N_k (w)$ with $k$ minimal so that $N_k (v) \neq N_k (w)$.
  
  Given $x \in X_{\infty}$ let $c (x, N)$ denote the total number of
  occurences of $\{ x_N, x_N^{- 1} \}$ in the word $R_N (x) .$ Define $\phi :
  G \rightarrow G_1 \times G_2 \ldots$ via $\phi (g) = \Pi_1 (g), \Pi_2 (g)
  \ldots$ and note $\phi$ is one to one. Thus $G$ is $T_2$ since the codomain
  is $T_2 .$ Consequently convergent sequences in $G$ have unique limits.
  
  Proof of Claim 1. Note for each $g_n \in G_n$ the corresponding subset $g_n
  \subset X$ has a unique minimal element $\sigma_n (g_n) \in X_n .$ To obtain
  a contradiction suppose $g \in G$ and $N \in \{ 1, 2, 3, \ldots \}$ is
  minimal so that $\{ R_N (\sigma_n \Pi_n (g)) \}$ is not eventually constant.
  Obtain $x \in g.$ Obtain $M$ so that if $M \leq n$ and $1 \leq k
  \leq N - 1$ then $\{ R_{N - k} (\sigma_n \Pi_n (g)) \}$ is constant.
  Thus for each $n \geq M$ we have $R_N (\sigma_n \Pi_n (g))$<$R_N
  (\sigma_{n + 1} \Pi_{n + 1} (g))$ or $R_N (\sigma_n \Pi_n (g))$=$R_N
  (\sigma_{n + 1} \Pi_{n + 1} (g))$. Thus, if $R_N (\sigma_n \Pi_n (g))$ is
  not eventually constant we have $c (R_N (\sigma_n \Pi_n (g)), N) \rightarrow
  \infty$, contradicting the fact that $c (x, N) \geq c (R_N (\sigma_n
  \Pi_n (g)), N)$ for all $n$.
  
  By definition if $g \in G$ and $x \in g$ then $R_N (x)$ and $R_N (\sigma
  (g)) $are $q_N $equivalent and hence $g \in \sigma (g) .$ To see that
  $\sigma$ is one to one, note if $\{ g, h \} \subset G$ with $g \neq h$, get
  $N$ minimal so that $\Pi_N (g) \neq \Pi_N (h)$, and note $\sigma_n (\Pi_N
  (g)) \neq \sigma_n (\Pi_N (h))$. \ Note, in $X_N$, $R_N (\sigma (g))$
  reduces to $\sigma_n (\Pi_N (g))$ and $R_N (\sigma (h))$ reduces to
  $\sigma_n (\Pi_N (h))$. 
Thus $R_N (\sigma (g)) \neq R_N (\sigma
  (h))$ and hence $\sigma (h) \neq \sigma (g) .$ Note $x_1 x_n x_1^{- 1}
  \rightarrow \emptyset \in G$and $\sigma (\emptyset) = \emptyset$.
  However $\sigma (x_1 x_n x_1^{- 1}) \rightarrow (x_1 x_1^{- 1}, x_1 x_1^{-
  1}, \ldots .) \neq \emptyset .$ Thus $\sigma$ is not continuous. This
  proves claim 1.
  
  To prove claim 2 suppose $\{ g_n \}$ is a convergent sequence in $G$. Since
  $\Pi_n$ is continuous and $G_N$ is discrete, the sequence $\Pi_N (g_n) $ is
  eventually constant. Since $X_{\infty}$ is metrizable to prove $\{ \sigma
  (g_n) \}$ has compact closure it suffices to prove $\overline{\{ \sigma
  (g_n) \}}$ is sequentially compact. Thus it suffices to prove, for each $N$,
  and allowing $n$ to vary, the set $\{ R_N (\sigma g_n) \}$ is finite. Fix
  $N$. To seek a contradiction, if $\{ R_N (\sigma g_n) \}$ were infinite then
  some subsequence is comprised of infinitely many distinct terms in the
  discrete space $X_N .$ Wolog we assume $\{ R_N (\sigma g_n) \}$ itself is
  comprised of infinitely many distinct terms. Note $c (R_N (\sigma g_n), N)
  \rightarrow \infty$. On the other hand since $q$ is a quotient map, some
  subsequence of $\{ g_n \}$ lifts to a convergent sequence in $X_{\infty} .$
  Thus $c (\ast, N)$ of the subsequential lifts is bounded. This contradicts
  the general fact that if $q (x) = g$ then $c (x, N) \geq c (\sigma (g),
  N),$ the latter inequality argued as follows. If we let $c (x, N, M)$ denote
  the total number of occurences of $\{ x_N, x_N^{- 1} \} $in $R_M (x),$ then
  by definition of $X_{\infty}$ we have $c (x, N, M) = 0$ if $M < N$ and $c
  (x, N, M) = c (x, N)$ if $N \leq M$. By definition of $\sigma$ obtain
  $M \geq N$ so that $R_N (\sigma_M (q_M (R_M x))) = R_N (\sigma q (x))$.
  Thus, $c (x, N) = c (x, N, M) \geq c (\sigma_M (q_M (R_M x)), N, M) = c
  (\sigma q (x), N, M) = c (\sigma q (x), N) .$
  
  For the converse of claim 2, suppose $\{ g_n \}$ is a sequence in $G$ such
  that $\{ \Pi_N (g_n) \}$ converges for each $N$ and such that $\{ \sigma
  (g_n) \}$ has compact closure. Since $X_{\infty}$ is metrizable,
  $\overline{\{ \sigma (g_n) \}}$ is sequentially compact. Let $x \in
  X_{\infty} $and $y \in X_{\infty}$ be subsequential limits of $\{ \sigma
  (g_n) \}$. With subsequences $v_n \rightarrow x$ and $w_n \rightarrow y.$It
  suffices to prove $q (x) = q (y)$, i.e. to prove $\Pi_N (q (x)) = \Pi_N (q
  (y))$ for all $N.$ Exploiting our hypothesis, continuity of $R_N$, and the
  fact that $X_N $is a discrete space, obtain $M$ so that $\Pi_N (q (x)) =
  q_{_{} N} (R_N x) = q_N (R_N (v_M)) = \Pi_N (g_M) = q_N (R_N (w_M)) = q_N
  (R_N (y)) = \Pi_N (q (y))$. This proves claim 2.

\end{proof}

\end{document}